\documentclass[11pt]{article}
\usepackage{amsmath,amsthm,amsfonts,amssymb,graphicx,graphpap}
\usepackage[hypertex]{hyperref}
\usepackage{xcolor}
\hypersetup{colorlinks=true,linkcolor=blue}
    
\usepackage[all]{xy}
\usepackage{thmtools, thm-restate}
\usepackage[a4paper,textwidth=15cm,textheight=25cm]{geometry}

\title{Translation lengths of outer automorphisms of finitely generated free-by-finite groups}
\author{Ioannis Papavasileiou and Mihalis Sykiotis}
\makeatletter
\newtheorem{theorem}{Theorem}[section]

\newtheorem{lemma}[theorem]{Lemma}
\newtheorem{corollary}[theorem]{Corollary}

\theoremstyle{definition}
\newtheorem{definition}[theorem]{Definition}
\theoremstyle{remark}
\newtheorem{remark}[theorem]{Remark}

 \def\classification{\@ifnextchar [{\@xfootnotetext}%
 {\begingroup\let\protect\noexpand\xdef\@thefnmark{}
 \endgroup\@footnotetext}}
 \def\keywords{\@ifnextchar [{\@xfootnotetext}%
 {\begingroup\let\protect\noexpand\xdef\@thefnmark{}
 \endgroup\@footnotetext}} 
 
\begin{document}

\classification {2020 {\sl Mathematics Subject Classification}:
20E06, 20E07, 20E36, 20F65.}


\maketitle

\begin{abstract}
Bestvina, Feighn and Handel proved that every subgroup of the outer
automorphism group, $\textrm{Out}(F_n)$, of the free group of rank $n$ is either virtually finitely generated abelian or contains a nonabelian free group. In this note we consider the more general situation
of the outer automorphism group $\textrm{Out}(G)$ of a finitely generated free-by-finite group $G$.
We show that $\textrm{Out}(G)$ is translation discrete and that every subgroup of $\textrm{Out}(G)$
is either virtually finitely generated abelian or contains a nonabelian free group.
\end{abstract}

\section{Introduction}

In $1972$, J. Tits \cite{Ti} proved that every subgroup of a finitely generated linear group  either contains
a free group of rank two or is virtually solvable. Since then a group is said to satisfy the \emph{Tits alternative} if each of its subgroups either contains a free group of rank two or is virtually solvable. Showing that various groups satisfy versions of the Tits alternative (in which subgroups that are not particularly ``special" must contain a non-abelian free group) has been a driving force in geometric group theory for decades. The \emph{strong Tits alternative} for a group $G$ states that every subgroup of $G$ which does not contain a non-abelian free group must be a finite extension of an  abelian group.  In \cite{BFH00,BFH05}, Bestvina, Feighn and Handel proved that $\textrm{Out}(F_n)$ satisfies the strong Tits alternative. In \cite{BFH04}, they complemented their result by showing that every solvable subgroup of $\textrm{Out}(F_n)$ has a finitely generated free abelian subgroup of finite index.
Besides outer automorphism groups of finitely generated free groups and linear groups, a ``Tits alternative" has been shown to hold for various classes of groups, such as hyperbolic groups \cite{Gr}, mapping class groups of compact surfaces \cite{Iv}, \cite{McC}, groups acting freely and properly on a $\textrm{CAT}(0)$ cube complex \cite{SW} and outer automorphism groups of certain right-angled Artin groups \cite{CV}. 
\par

The notion of \emph{translation numbers} in a metrized group is motivated by the fact that for an isometry $g$ of a hyperbolic space $\mathbb{H}$ the infimum of displacements $\tau(g)=\inf\{d(gx,x)|x\in\mathbb{H}\}$ can be computed as $\tau(g)=\lim\limits_{n\to\infty}\frac{d(g^nx_0,x_0)}{n}$ for any choice of fixed base point $x_0$. In many geometric cases there exist a set of points which $g$ displaces minimally and $g$ \emph{translates} this set by the distance $\tau(g)$. In the literature
this notion of translation has been extended to all metric groups by replacing the
hyperbolic distance by the metric distance in the group, very often the \emph{word metric} of a finitely generated group. 

In \cite{GS}, Gersten and Short  proved a number of results about biautomatic groups using translation numbers. In \cite{Gr} Gromov proved that translation numbers corresponding to a word metric in a word hyperbolic group are rational with bounded denominator and thus discrete. The condition of \emph{translation discrete} was introduced
by Conner to study nilpotent and solvable groups acting on non-positively curved
spaces. It states that the positive translation numbers of a group $G$ are bounded
away from zero. It should be noted that translation discreteness was used by Alibegovic \cite{Al}, who proved that $\textrm{Out}(F_n)$ is translation discrete, to give an alternative short proof of the celebrated theorem of Bestvina, Feighn and Handel which says that every solvable subgroup of $\textrm{Out}(F_n)$ is finitely generated and virtually abelian.\par

In this note, we first obtain a generalization  of Alibegovic's result for outer automorphisms of finitely generated free-by-finite groups.

\begin{restatable*}{theorem}{firstThmOne}\label{ThmOne} 
    Let $G$ be a finitely generated free-by-finite group. Then $\textrm{Out}(G)$ is translation discrete and satisfies the strong Tits alternative.
 \end{restatable*}
In \cite{Con00}, Conner combining the algebraic and geometric aspects of translation discreteness proved that:
\begin{theorem}(\cite[Theorem 3.4]{Con00})
Every solvable subgroup of finite virtual cohomological dimension $m$, in a translation discrete group is a finite extension of $\mathbb{Z}^m$.
\end{theorem}

Making use of the above theorem we extend the above-mentioned result for $\textrm{Out}(F_n)$ to the outer automorphism groups of finitely generated free-by-finite groups.

\begin{restatable*}{theorem}{firstThmTwo}\label{thm2}
  Let $G$ be a finitely generated free-by-finite group. Then  every subgroup of $\textrm{Out}(G)$  (finitely generated or not) either contains a free subgroup of rank two or is virtually $\mathbb{Z}^m$, where $m\leq vcd[ \textrm{Out}(G)]$. 
\end{restatable*}

\textbf{Acknowledgements}: The authors would like to thank the referee for many helpful suggestions which improved the exposition of the paper.

\section{Preliminaries}

We start by recalling the definition and some basic facts about translation lengths 
(see for example \cite{GS}).
Let $G$ be a finitely generated group and $X$ a finite generating set for $G$. For each
$g\in G$, we denote by $\|g\|_X$ the word length of $g$ with respect to $X$, i.e. the minimum integer $n$ such that $g$ can be expressed as a product $g=s_{1}\cdots s_{n}$, where $s_{i}\in X^{\pm 1}$ for each 
$i=1,\ldots,n$. We will write $\|g\|_G$ (resp. $\|g\|_X$) instead of $\|g\|_X$ if $X$ 
(resp. $G$) is well understood from context.
\begin{definition}
Let $G$ be a finitely generated group and $X$ a finite generating set for $G$.
The \emph{translation length}  with respect to $X$ of an element $g\in G$, denoted $\tau_{\scriptscriptstyle{G,X}}(g)$, is
the limit
\[
\tau_{\scriptscriptstyle{G,X}}(g)=\lim_{n\to \infty}\dfrac{\|g^n\|_{X}}{n}.
\]
\end{definition}

The existence of the above limit follows by the subadditivity of $\|\cdot\|_X$, that is 
$\|gh\|_X\leq \|g\|_X +\|h\|_X$ for all $g,h\in G$. We will write $\tau_{\scriptscriptstyle{G}}(g)$ instead of $\tau_{\scriptscriptstyle{G,X}}(g)$ when the generating set $X$ is understood.\par

In the following lemma we collect some elementary properties of translation numbers (we refer to \cite{GS} for proofs).
 \begin{lemma}
 Let $G$ be a finitely generated group and $X$ a finite generating set for $G$.
 \begin{enumerate}
     \item If $g\in G$ is an element of finite order, then $\tau_{\scriptscriptstyle{G,X}}(g)=0$. 
     \item $\tau_{\scriptscriptstyle{G,X}}(g)=\tau_{\scriptscriptstyle{G,X}}(xgx^{-1})$ for all $g,x\in G$.
     \item $\tau_{\scriptscriptstyle{G,X}}(g^n)=|n|\tau_{\scriptscriptstyle{G,X}}(g)$  for all $g\in G$ and $n\in\mathbb{Z}$.
     \item If $x,y\in G$ are commuting elements, then $\tau_{\scriptscriptstyle{G,X}}(xy)\leq \tau_{\scriptscriptstyle{G,X}}(x)+\tau_{\scriptscriptstyle{G,X}}(y)$.
     \item $\tau_{\scriptscriptstyle{G,X}}(g)=\tau_{\scriptscriptstyle{G,X}}(g^{-1})$ for all $g\in G$.
     \item Let $G,H$ be two finitely generated groups and $X$ a finite generating set for $G$. If $f: G\to H$ is an epimorphism and $Y$ is the image of $X$ under $f$, then for all $g\in G$ we have $\tau_{\scriptscriptstyle{G,X}}(g)\ge \tau_{\scriptscriptstyle{H,Y}}(f(g))$.
 \end{enumerate}
 \end{lemma}

\begin{definition}
A finitely generated group $G$ is called \emph{translation discrete} if for some generating set $X$ of $G$ the translation numbers of the non-torsion elements are bounded away from zero, i.e. there is a positive constant $M$ such that $\tau_{\scriptscriptstyle{G}}(g)\ge M$ for all elements $g\in G$ of infinite order.
\end{definition}

The following lemma says that homomorphisms which are
also quasi-isometries can only change translation numbers by multiplicative
constants so that zero translation numbers and translation discreteness
are preserved. The proof is left as an exercise to the reader.

\begin{lemma}\label{LemHomQi}
Let $G,H$ be two finitely generated groups, let $X,Y$ be finite generating sets for $G$ and $H$, respectively, and let $f:G\to H$ be a homomorphism. If $f$ is a quasi-isometry, then there is a positive number $\lambda$ such that 
\[
\frac{1}{\lambda}\tau_{\scriptscriptstyle{G,X}}(g)\leq \tau_{\scriptscriptstyle{H,Y}}(f(g))\leq \lambda\tau_{\scriptscriptstyle{G,X}}(g), \text{ for all } g\in G.
\]
In particular:
\begin{enumerate}
\item The  translation discreteness of $G$ does not depend on the generating set $X$.
\item $G$ is translation discrete if and only if each subgroup of finite index in $G$ is translation discrete.
\item If $K$ is a finite normal subgroup of $G$, then $G$ is translation discrete if and only if $G/K$ is translation discrete.
\end{enumerate}
\end{lemma}
It is easy to see that a homomorphism $f:G\to H$ as in the above lemma is a quasi-isometry if and only if $\ker f$ is finite and $\textrm{Im} f$ is of finite index in $H$.
\begin{lemma}[\cite{Con98} Lemma 2.6.1]\label{LemSubdiscrete}
Let $G$ be a finitely generated group and $X$ a finite generating set. The following are equivalent:
\begin{enumerate}
    \item $G$ is translation discrete.
    \item Every finitely generated subgroup of $G$ is translation discrete.
\end{enumerate}
\end{lemma}

\section{The main theorems}\label{Free-By-Finite}
Throughout this section, we assume that $\mathcal{P}$ is a group theoretic property satisfying:
\begin{enumerate}
\item $\mathcal{P}$ is subgroup closed.
\item If $H$ is a subgroup of finite index in a group $G$ and $H$ has $\mathcal{P}$, then $G$ has $\mathcal{P}$.
\item If $G$ has  $\mathcal{P}$ and $N$ is a finite normal subgroup of $G$, then $G/N$ has $\mathcal{P}$.
\end{enumerate}

Examples of such properties are: the property of being virtually torsion free, the property of having finite virtual cohomological dimension, and the property of being translation discrete for finitely generated groups (Lemmas \ref{LemHomQi} and \ref{LemSubdiscrete}).%, and the property of satisfying the %Tits %alternative (Lemma \ref{Lem4}).

\begin{definition}
 Let $G$ be a group.  We say that: 
\begin{enumerate}
    \item  $G$ satisfies the \emph{Tits alternative} if every subgroup of $G$ (finitely generated or not) is either virtually solvable, or contains a free subgroup of rank $2$.
    \item  $G$ satisfies the \emph{strong Tits alternative} if every subgroup of $G$  (finitely generated or not) is either virtually abelian, or contains a rank two free subgroup.
\end{enumerate}
 \end{definition}

\begin{lemma}\label{Lem4}
The property of satisfying the Tits alternative (resp. the strong) is a group theoretic property $\mathcal{P}$ as above.
 \end{lemma}
 \begin{proof}
 It is clear that the class of groups satisfying the Tits alternative (resp. the strong) is closed under taking subgroups and under finite index supergroups. 
Hence,  it will suffice to show that if $G$ satisfies the Tits alternative (resp. the strong) and $K$ is a finite normal subgroup of $G$, then $G/K$ satisfies the Tits alternative (resp. the strong), too. \par
Each subgroup $\overline{H}$ of  $G/K$ has the form $H/K$, where $H$ is a subgroup of $G$ containing $K$. If $H$ contains a free subgroup $F$ of rank $2$, then, since $K$ is finite, $F$ embeds in the quotient group $\overline{H}=H/K$ and thus $\overline{H}$ contains a free subgroup of rank $2$. If $H$ is virtually solvable (resp. abelian), then it contains a solvable (resp. abelian) subgroup $Q$ of finite index and the quotient group  $QK/K$ is solvable (resp. abelian) being isomorphic to $ Q/(Q\cap K)$. Also, it is of finite index in $H/K$, since $$\left|\frac{H}{K}:\frac{QK}{K}\right|=|H:QK|\leq |H:Q|<\infty.$$ 
  \end{proof}

\begin{lemma}\label{Lem3}
Let $G$ be a finitely generated group, $N$ a normal subgroup of $G$ of finite index with trivial center and $\mathcal{P}$ a group theoretic property as above. 
If $\textrm{Out}(N)$ satisfies $\mathcal{P}$, then so does $\textrm{Out}(G)$. In particular, if $\textrm{Out}(N)$ satisfies the Tits alternative (resp. the strong), then $\textrm{Out}(G)$ satisfies the Tits alternative (resp. the strong), too.
 \end{lemma}
 \begin{proof}
 Let $\textrm{Aut}_N(G)$ denote the subgroup of $\textrm{Aut}(G)$ consisting of those automorphisms of $G$ which fix $N$ and induce the identity on $G/N$. The hypothesis that $N$ has trivial center implies
that the restriction map $\phi:\textrm{Aut}_N(G)\to \textrm{Aut}(N)$ is an injection and hence $\phi(\textrm{Inn}_N(G))=\textrm{Inn}(N)$, where $\textrm{Inn}_N(G)$ denotes the normal subgroup of $\textrm{Inn}(G)$ consisting of all inner automorphisms of $G$ induced by elements of $N$. It follows that the quotient group $\Gamma=\textrm{Aut}_N(G)/\textrm{Inn}_N(G)$ embeds into $\textrm{Out}(N)$ and therefore $\Gamma$ satisfies $\mathcal{P}$ . \par
Since $G$ is finitely generated, there are finitely many subgroups of index $[G:N]$ and thus the set $X=\{\varphi(N), \varphi \in \textrm{Aut}(G) \}$ is finite. 
If $K$ is the kernel of the natural action of  $\textrm{Aut}(G)$ on $X$, then the index $[\textrm{Aut}(G):K]$ is finite and each automorphism $\varphi \in K$ induces an automorphism 
$\tilde{\varphi}:G/N\rightarrow G/N$. So we obtain a homomorphism $\psi: K\rightarrow \textrm{Aut}(G/N)$, whose kernel is precisely  $\textrm{Aut}_N(G)$. Since $\textrm{Aut}(G/N)$ is finite, it follows that $\textrm{Aut}_N(G)$ is of finite index in $K$ and hence in $ \textrm{Aut}(G)$.\par
The restriction of the natural projection $\pi:\textrm{Aut}(G)\to \textrm{Out}(G)$ to $\textrm{Aut}_N(G)$ induces a homomorphism $\Bar{\pi}:\Gamma \to \textrm{Out}(G)$. The kernel of $\Bar{\pi}$, which is equal to $\textrm{Inn}(G)/\textrm{Inn}_N(G)$, is finite because $N$ is of finite index in $G$.  It follows that the image $\textrm{Im}\Bar{\pi}$ satisfies $\mathcal{P}$. But  $\textrm{Im}\Bar{\pi}$ has finite index in $\textrm{Out}(G)$, since $\textrm{Aut}_N(G)$ is of finite index in $\textrm{Aut}(G)$, which implies that $\textrm{Out}(G)$ satisfies $\mathcal{P}$. The second assertion follows from Lemma \ref{Lem4}.
\end{proof}

\begin{remark}\label{fgcase} From the above proof it also follows that $\Gamma$ is finitely generated if and only if $\textrm{Out}(G)$ is.
Thus, in the case where $\mathcal{P}$ is defined only for finitely generated groups, the previous lemma still holds if we further assume that $\textrm{Out}(G)$ is finitely generated.
\end{remark}

 \firstThmOne
  \begin{proof}
  Since $G$ is a finitely generated, free-by-finite group, it contains a finitely generated normal free subgroup $F$ of finite index. 
  If the rank of $F$ is $1$, then  $\textrm{Out}(G)$ is finite and trivially satisfies the conclusion.\par
  If the rank of $F$ is greater than $1$, then, since $\textrm{Out}(G)$ is finitely generated \cite{Kr} and $\textrm{Out}(F)$ is translation discrete \cite{Al}, it follows from Lemma \ref{Lem3} and Remark \ref{fgcase} that $\textrm{Out}(G)$ is also translation discrete. Appealing to Lemma \ref{Lem3} completes the proof.
  \end{proof}
   
  It is well known that the outer automorphism group of a finitely generated free-by-finite group has finite virtual cohomological dimension (see for example \cite{Kal}, \cite{mCool}). Furthermore Conner proved  \cite[Theorem 3.4]{Con00}  that every solvable group of finite virtual cohomological dimension $m$ in a translation discrete group is virtually $\mathbb{Z}^m$ (see also \cite{Bes}). 
Combining these with Theorem \ref{ThmOne} we get the following: 

\firstThmTwo

If $G$ is a finitely generated free-by-finite group, then Lemma \ref{Lem3} can also be used to show that each subgroup of $\textrm{Out}(G)$ either contains a free group of rank 2 or is virtually finitely generated abelian. However, translation discreteness provides an upper bound for the torsion free rank of every abelian subgroup of $\textrm{Out}(G)$.\par

 The Tits alternative for outer automorphism groups of free products has been studied by C. Horbez in \cite{Ho}. Krsti\'{c} and Vogtmann \cite{KV} computed the virtual cohomological dimension of the outer automorphism group of a finitely generated group which is a free product of a free group and finite groups. As a corollary of the above theorem and \cite[Proposition 10.1]{KV}, we obtain:
  \begin{corollary}
  Let $G=G_1\ast\cdots\ast G_k\ast F_p$, where each free factor $G_i$ is a finite group and $F_p$ is a free group of rank $p$. Then every subgroup of $\textrm{Out}(G)$ either contains a free group of rank two or is virtually  $\mathbb{Z}^m$, where $m\leq 2p+k-2$.
  \end{corollary}

\noindent
Department of Mathematics\\
National and Kapodistrian University of Athens\\
Panepistimioupolis, GR-157 84, Athens, Greece\\
{\it e-mail}: iopapav@math.uoa.gr\\
{\it e-mail}: msykiot@math.uoa.gr

\end{document}